\documentclass[12pt,a4paper]{article}
\usepackage[latin1]{inputenc}
\usepackage[UKenglish]{babel}
\usepackage{csquotes}
\usepackage[UKenglish]{datetime}
\usepackage[numbers,sort]{natbib}
\usepackage[normalem]{ulem}

\usepackage{amsmath,amsthm,amsfonts,amssymb}
\usepackage{graphicx}
\usepackage{lmodern}
\usepackage[margin=1.2in]{geometry}
\usepackage{mathtools}

\usepackage[inline, shortlabels]{enumitem}

\usepackage{color}

\newdateformat{UKvardate}{
	\textit{\monthname[\THEMONTH] \THEYEAR}}

\newcommand{\absto}{\mathcal{O}}
\newcommand{\pow}{\mathsf{P}}
\newcommand{\univ}{\mathfrak{U}}
\newcommand{\uni}{\mathsf{U}}

\newcommand{\lc}{^>}
\newcommand{\uc}{^<}

\newtheorem{definition}{Definition}
\newtheorem{theorem}[definition]{Theorem}
\newtheorem{lemma}[definition]{Lemma}
\newtheorem{corollary}[definition]{Corollary}

\makeatletter
\newcommand{\superimpose}[2]{{\ooalign{$#1\@firstoftwo#2$\cr\hfil$#1\@secondoftwo#2$\hfil\cr}}}
\makeatother

\DeclareMathOperator{\bigveeplus}{\mathchoice{\mathpalette\superimpose{{\bigvee}{\raisebox{6pt}{\tiny +}}}}{\mathpalette\superimpose{{\bigvee}{\raisebox{5pt}{\scalebox{.5}{+}}}}}{\mathpalette\superimpose{{\bigvee}{\raisebox{5pt}{\scalebox{.4}{+}}}}}{\mathpalette\superimpose{{\bigvee}{\raisebox{5pt}{\scalebox{.4}{+}}}}} }

\makeatletter
\newcommand{\inoreq}{\mathrel{\text{\in@eq}}}
\newcommand{\in@eq}{%
	\oalign{%
		\hidewidth$\m@th\in$\hidewidth\cr
		\noalign{\nointerlineskip\kern1ex}%
		$\m@th\smash{-}$\cr
		\noalign{\nointerlineskip\kern-.5ex}%
	}%
}
\makeatother

\author{Zurab Janelidze and Ineke van der Berg}
\title{A Dedekind-style axiomatization and the corresponding universal property of an ordinal number system}
\date{}

\begin{document}
	\maketitle
	
	\begin{abstract}
	In this paper, we give an  axiomatization of the ordinal number system, in the style of Dedekind's axiomatization of the natural number system. The latter is based on a structure $(N,0,s)$ consisting of a set $N$, a distinguished element $0\in N$ and a function $s\colon N\to N$. The structure in our axiomatization is a triple $(O,L,s)$, where $O$ is a class, $L$ is a class function defined on all $s$-closed `subsets' of $O$, and $s$ is a class function $s\colon O\to O$. In fact, we develop the theory relative to a Grothendieck-style universe (minus the power set axiom), as a way of bringing the natural and the ordinal cases under one framework. We also establish a universal property for the ordinal number system, analogous to the well-known universal property for the natural number system. 
	\end{abstract}
	
	
	
	
	\section*{Introduction}
	
	The introduction and study of ordinal numbers goes back to the pioneering works of Cantor in set theory \cite{cantor1895beitrage,cantor1897beitrage}. 
	In modern language, Cantor's ordinal numbers are isomorphic classes of well-ordered sets, see e.g.~\cite{fraenkel1953abstract}. There is also a `concrete' definition of an ordinal number as a transitive set which is well-ordered under the element relation -- see e.g.~\cite{jech2003set}. Such sets are usually called \emph{von Neumann ordinals}. Natural numbers can be seen concretely as the finite ordinal numbers. 
	In Dedekind's approach to the natural number system described in \cite{dedekind1888sollen}, the natural numbers are not defined as concrete objects, but rather as abstract entities organized in a certain structure; namely, a triple $(N,0,s)$ consisting of a set $N$ (a set of `abstract' natural numbers), a distinguished element $0$ of $N$ (in \cite{dedekind1888sollen}, the distinguished element is 1), and a function $s\colon N\to N$, which names the `successor' of each natural number. The axioms that such a system should satisfy were formulated by Dedekind, but are often referred to as Peano axioms today (see e.g.~\cite{wang1957axiomatization} for some historical background):
	\begin{itemize}
	    \item $0$ does not belong to the image of $s$.
	    
	    \item $s$ is injective.
	    
	    \item $X=N$ for any subset $X$ of $N$ that is closed under $s$ and contains $0$.
	\end{itemize}
	It is an observation due to Lawvere \cite{lawvere1964elementary} that these axioms identify the natural number systems as initial objects in the category of all triples $(X,x,t)$ where $X$ is a set, $x\in X$ and $t$ is a function $t\colon X\to X$. This `universal property' of the natural number system, freed from its category-theoretic formulation, is actually the `definition by induction' theorem already contained in \cite{dedekind1888sollen}. A morphism $(N,0,s)\to(X,x,t)$ of such triples is defined as a function $f\colon N\to X$ such that:
	\begin{itemize}
	    \item $f(0)=x$,
	    \item $f(s(n))=t(f(n))$ for all $n\in N$.
	\end{itemize}
	The `definition by induction' theorem states that there is exactly one morphism to any other triple $(X,x,t)$ from the triple $(N,0,s)$ satisfying the axioms stated above. This theorem is of course well known because of its practical use: it says that recursively defined functions exist and are uniquely determined by the recursion. Intuitively, the theorem can be understood as follows. A triple $(X,x,t)$ can be viewed as an abstraction of the concept of counting -- $X$ is the set of figures used in counting, $x$ is where counting begins and the function $t$ names increments when counting. Without further restrictions on such `counting systems', there are many non-isomorphic ones, some of which are quite different from the natural number system, but still useful; for instance, hours on a clock, where counting loops back to $1$ once we pass $12$. A morphism of these triples can be viewed as a `translation' of one counting system to another. The universal property of the natural number system presents it as a `universal' counting system, in the sense that it has a unique translation to any other counting system. Incidentally, such intuition is not particular to the natural number system: many structures in mathematics can be defined by natural universal properties -- see \cite{mac2013categories}. 
	
	Ordinal numbers exhibit a similar structure to natural numbers -- there is a `starting' ordinal (the natural number $0$), and every ordinal number has a successor. The natural numbers $0,1,2,3,\ldots$ are the first ordinal numbers. This set is closed under succession. The ordinal number system allows for another type of succession that can be  applied to sets of ordinal numbers closed under succession, giving rise to the so-called `limit' ordinal numbers. The infinite sequence of natural numbers is succeeded by a limit ordinal number, usually denoted by $\omega$. Now, we can take the `usual' successor of $\omega$, call it $\omega+1$, and keep taking its successors until we get another set that is closed under succession, after which we introduce another limit ordinal number -- it will be $\omega+\omega=\omega\cdot 2$. The next limit ordinal number will be $\omega\cdot 3$. At some point, we reach $\omega\cdot\omega=\omega^2$, then $\omega^3$, and so on until we reach $\omega^\omega$, then $\omega^\omega+1$, and so on\ldots The process is supposed to continue until all ordinal numbers that we have named no longer form a set. 
	Let us also recall that von Neumann ordinals are defined as sets of preceding ordinals. Thus the first ordinal number, the number $0$, is defined as $0=\varnothing$ and the successor of an ordinal number $n$ is defined as $\uni\{n, \{n\}\}$. A limit ordinal number is one that is the union of all preceding ordinal numbers. We can, in particular, think of $0$ as a limit ordinal number given by the union of its predecessors, since the empty union equals the empty set. Equivalently, limit ordinal numbers are those whose sets of predecessors are closed under succession. 
	
	We may also think of the ordinal number system in terms of a triple $(O,L,s)$ -- this time, $O$ is a class (since the collection of all ordinal numbers is no longer a set), $L$ is a (class) function that specifies limit ordinals and is defined for those subclasses of $O$ which form sets closed under the class function $s$, which specifies the successor of each ordinal. In this paper we show that the following three axioms on such a triple are suitable as analogues of the three Dedekind-Peano axioms for the ordinal number system:
	\begin{itemize}
	    \item $L(I)$ does not belong to the image of $s$, and also $s(L(I))\notin I$, for any $I$ such that $L(I)$ is defined. 
	    
	    \item $s$ is injective and $L(I)=L(J)$ if and only if  $\overline{I}=\overline{J}$ and $L(I),L(J)$ are defined, where $\overline{I}$ and $\overline{J}$ denote closures of $I$ and $J$, respectively, under $s$ and $L$ predecessors.
	    
	    \item $X=O$ for any subclass $X$ of $O$ that is closed under $s$ and that contains $L(I)$ for each $I\subseteq X$ such that $L(I)$ is defined. 
	\end{itemize}	
	In particular, we prove that:
	\begin{itemize}
	    \item The system of von Neumann ordinals constitutes a triple $(O,L,s)$ satisfying the three axioms above. There is nothing surprising here, as the result relies on the well-known properties of ordinal numbers.
	    
	    \item Any triple $(O,L,s)$ satisfying these three axioms has an order which makes it order-isomorphic to the system of ordinal numbers. The order, in fact, is the specialization order of the topology given by the closure operator in the second axiom (without those axioms, this order is merely a preorder). 
	    
	    \item The triple  $(O,L,s)$ satisfying the three axioms above is an initial object in the category of all triples $(O',L',s')$ such that $\overline{I'}=\overline{J'}$ implies $L(I')=L(J')$, whenever those are defined (with the domain of $L'$ being the class of all $s'$-closed subsets of $O'$). 
	\end{itemize}
	
	The idea for defining an ordinal number system abstractly goes back to Zermelo (see `Seven notes on ordinal numbers and large cardinals' in \cite{zermelo2010ernst}). His approach is to define it as a particular type of well-ordered class $(O,\leqslant)$. The following axioms would suffice:
	\begin{itemize}
	    \item For every $x\in O$, the class $\{x\}\lc =\{y\in O\mid y<x\}$ is a set.
	    
	    \item For each subset $X$ of $O$, the class $X\uc =\{y\in O\mid \forall_{x\in X}[x<y]\}$ is non-empty.
	\end{itemize}
	There is, of course, an analogous presentation (although less known than the one given by Peano axioms) of the natural number system as a well-ordered set $(N,\leqslant)$  satisfying the following conditions:  
	\begin{itemize}
	    \item For each $x\in N$, the set $\{x\}\lc $ is finite.
	    
	    \item For each $x\in N$, the set $\{x\}\uc $ is non-empty.
	\end{itemize}
	This is an alternative approach to that of Dedekind, where there is greater emphasis on the order structure. The difference between our approach to ordinal numbers and the traditional approaches is similar, where in our approach we try to make minimal use of the order structure. Universal properties of the ordinal number system emerging from the more order-based approach have been established in \cite{joyal1995algebraic}. 
	Our universal property is different from those. 
	
	The main new results of the paper are given in the last two sections. Before that, we redevelop the basic theory of ordinal numbers relative to the set-theoretic context in which these results are proved, ensuring that the paper is self-contained.

	\section{The context}

    There are a number of alternatives for a context in which the theory that we lay down in this paper could be developed. Elaboration of those contexts and comparison of results across the contexts as a future development of our work would certainly be worthwhile. In this paper, we have decided to stick to what we believe to be technically the most simple and intuitive context, given by a `universe' inside the standard Zermelo-Fraenkel axiomatic set theory, including the axioms of foundation and choice (see e.g.~\cite{jech2003set}). Developing mathematics relative to a universe is typical in those subjects where sets of different sizes are needed. For instance, this is the approach followed in the exposition of category theory in \cite{mac2013categories}. The universes we work with, however, are slightly more general than the more commonly used Grothendieck universes \cite{bourbaki2006theorie,gabriel1953des,williams1969grothendieck}. The main difference is that our universes do not require closure under power sets and can be empty. Our context is in fact a particular instance of the quite general category-theoretic context used in \cite{joyal1995algebraic}. Generalization of our results to that context is left for future work.
    
    We remark that the definitions, theorems and their proofs contained in this paper could be adapted, after a straightforward modification to their formulation to `absolute' set theory, where our `sets' could be replaced with `classes' and elements of the fixed universe with `sets'. We would then get the form of the definitions and theorems given in the Introduction. 
    
    For a set $X$, by $\pow X$ we denote the power set of $X$, and by $\uni X$ we denote the union of all elements of $X$. By $\mathbb{N}$ we denote the set of natural numbers. While we do not rely on any prior knowledge of facts about ordinal numbers (proofs of all needed facts are included in the paper), we do rely on knowledge of basic set-theoretic properties of the natural number system. In particular, we will make use of mathematical induction, definition by recursion, as well as the fact that any infinite set has a subset bijective to $\mathbb{N}$.

    Recall that a set $X$ is said to be \emph{transitive} when $X\subseteq \pow X$, or equivalently, when $\uni X\subseteq X$. For a function $f\colon X\to Y$ and a set $A\in\pow X$, we write $fA$ to denote the image of $A$ under $f$:
    \begin{gather*}
        fA=\{f(a)\mid a\in A\}.
    \end{gather*}

	
	\begin{definition}\label{def:univ}
		A \emph{universe} is a set $\univ$ satisfying the following:
		\begin{itemize}
			\item[(U1)] $\univ$ is a transitive set.
			
			\item[(U2)] If $X,Y\in\univ$ then $\{X,Y\}\in\univ$. 
			
			\item[(U3)] $\uni fI \in \univ$ for any $I \in \univ$ and any function $f\colon I\to \univ$.
		\end{itemize}

	\end{definition}
	
These axioms imply that $\univ$ is closed under the following standard set-theoretic constructions: 
\begin{itemize}
    \item Singletons. Trivially, since $\{x\}=\{x,x\}$.
    
    \item Union. Because $\uni X=\uni 1_X X$.
    
    \item Subsets. Let $X\subseteq Y$. If $Y\in\univ$ and $X=\varnothing$, then $X\in\univ$ because $\univ$ is a transitive set (thanks to the axiom of foundation). If $X\neq \varnothing$ then let $x_0\in X$. Consider the function $f\colon Y\to \univ$ defined by
    \begin{gather*}
    f(y)=\begin{cases}
    \{y\}, & y\in X,\\
    \{x_0\}, & y\notin X,
    \end{cases}
    \end{gather*}
    Then $X=\uni fY$ and so $X\in\univ$.
    
    \item Cartesian products (binary). Let $X\in\univ$ and $Y\in\univ$. Then \begin{gather*}\{(x,y)\}=\{\{\{x,y\},\{x\}\}\}\in\univ\end{gather*} for each $x\in X$ and $y\in Y$. For each $x\in X$ define a function $f_x\colon Y\to \univ$ by $f_x(y)=\{(x,y)\}$.
    Then 
    $\{(x,y)\mid y\in Y\}=\uni f_x Y\in\univ
    $
    for each $x\in X$. Now define a function $g\colon X\to \univ$ by $g(x)=\{\{(x,y)\mid y\in Y\}\}$. Then, $X\times Y=\uni g X\in\univ$.
    
    \item Disjoint union. Given $X\in\univ$, the disjoint union $\Sigma X$ can be defined as \begin{gather*}\Sigma X=\uni\{x\times\{x\}\mid x\in X\}.\end{gather*} Then $\Sigma X=\uni fX$ where $f\colon X\to \univ$ is defined by $f(x)=x\times\{x\}$.
    
    \item Quotient sets. Let $X\in\univ$ and let $E$ be an equivalence relation on $X$. Then $X/E=\uni qX$, where $q$ is the function $q\colon X\to \univ$ defined by $q(x)=\{[x]_E\}$. 
\end{itemize}
From this it follows of course that when $\univ$ is not empty, it contains all natural numbers, assuming that they are defined by the recursion
\begin{align*}
    0 &=\varnothing,\\
    n+1 &=\uni\{n,\{n\}\}.
\end{align*}
Furthermore, when $\univ$ contains at least one infinite set, it also contains the set $\mathbb{N}$ of all natural numbers (as defined above).

The empty set $\varnothing$ is a universe. The sets whose transitive closure have cardinality less than a fixed infinite cardinal $\kappa$ form a universe in the sense of the definition above (by Lemma~6.4 in \cite{kunen1980set}). In particular, hereditarily finite sets form a universe, as do hereditarily countable sets. The so-called Grothendieck universes are exactly those universes in our sense, which are closed under power sets, i.e.\ if $X\in\univ$, then $\pow X\in \univ$. 
	
For any two sets $A$ and $B$, we write
\begin{gather*}
    A\approx B
\end{gather*}
when there is a bijection from $A$ to $B$. The \emph{restricted power set} of a set $X$ relative to a universe $\univ$ is the set
\begin{gather*}
	\pow_\univ X = \{A\subseteq  X\mid \exists_{B\in \univ}[A\approx B]\}.
\end{gather*} 
The following lemmas will be useful.		

\begin{lemma}\label{lem:fini}
       When $\univ$ is non-empty, for any set $X$ and its finite subset $Y\subseteq X$, we have $Y\in \pow_\univ X$.
\end{lemma}

\begin{proof}
This follows from the fact that when $\univ$ is non-empty, it contains a set of each finite size. 
\end{proof}		

\begin{lemma}\label{lem:subs}
If $A\subseteq B$ and $B\in \pow_\univ X$, then $A\in \pow_\univ X$.
\end{lemma}

\begin{proof}
If $A\subseteq B\approx C\in\univ$, then $A$ is bijective to a subset of $C$. 
\end{proof}

\begin{lemma}\label{lem:unio}
If $A,B\in \pow_\univ X$ then $\uni\{A,B\}\in \pow_\univ X$. Also, if $I\in \pow_\univ X$, then for any function $f\colon I\to \pow_\univ X$, we have: $\uni fI\in \pow_\univ X$. 
\end{lemma}

\begin{proof}
Let $A\approx A'$ and $B\approx B'$ where $A',B'\in\univ$. Then $\uni\{A,B\}$ is bijective to a suitable quotient set of a disjoint union of $A'$ and $B'$. Next, to prove the second part of the lemma, suppose $I\approx J\in\univ$ and for each $x\in I$, suppose $f(x)\approx g(x)\in\univ$. Let $h$ denote a bijection $h\colon J\to I$. Define a function $f'\colon J\to\univ$ by $f'(x)=\{gh(x)\}$. Then $fI\approx \uni f'J$. The union $\uni fI$ will then be bijective to a suitable disjoint union of $\uni f'J$. 
\end{proof}
	

\begin{lemma}\label{lem:imag}
Given a function $f\colon X\to Y$,
	\begin{gather*}
	A\in\pow_\univ X\quad\Rightarrow\quad fA\in\pow_\univ Y.
	\end{gather*}
\end{lemma}

\begin{proof}
If $A\approx A'\in\univ$, then $fA$ is bijective to a suitable quotient of $A'$.
\end{proof}

\begin{lemma}\label{lem:produ}
If $A\in\pow_\univ X$ and $B\in\pow_\univ Y$, then $A\times B\in\pow_\univ (X\times Y)$.
\end{lemma}

\begin{proof}
This follows from the fact that $\univ$ is closed under cartesian products.
\end{proof}
	
\begin{lemma}\label{lem:com} $\pow_\univ \univ\subseteq\univ$. 
	\end{lemma}
	
	\begin{proof} Let $A\approx B$, $A\subseteq\univ$ and $B\in\univ$. Write $f$ for a bijection $f\colon B\to A$. Consider the function $g\colon B\to\univ$ defined by $g(b)=\{f(b)\}$. Then $A=\uni gB$ and so $A\in\univ$.
	\end{proof}

	\section{Abstract ordinals}
	
	In this section we introduce an abstract notion of an ordinal number system relative to a universe and establish its basic properties. Consider a partially ordered set $(X,\leqslant)$. The relation $<$ for the partial order, as a relation from $X$ to $X$, induces a Galois connection from $\pow X$ to itself given by the mappings
	\begin{align*}
	S\mapsto S\lc  &= \{x\in X\mid \forall_{y\in S} [x<y]\},\text{ and}\\
	S\mapsto S\uc  &= \{x\in X\mid \forall_{y\in S} [y<x]\}.
	\end{align*}
	We call $S\lc $ the \emph{lower complement} of $S$, and $S\uc $ the \emph{upper complement} of $S$. Note that by `$a<b$' above we mean `$a\leqslant b\land a\neq b$', as usual. Since the two mappings above form a Galois connection, we have
	\begin{gather*}
	    S\subseteq (S\lc )\uc  \textrm{ and } S\subseteq (S\uc )\lc 
	\end{gather*}
	for any $S\subseteq X$. We define the \emph{incremented join} $\bigveeplus S$ of a subset $S$ of $X$ (when it exists) as follows:
	\begin{gather*}
	\bigveeplus S = \min(S\uc ).
	\end{gather*}
	Note that since $S\cap S\uc =\varnothing$, the incremented join of $S$ is never an element of $S$. 
	It will be convenient to use the following abbreviations (where $x\in X$ and $S\subseteq X$):
	\begin{gather*}
	    x^+=\bigveeplus\{x \},\quad S^+=\{x^+\mid x\in S\}.
	\end{gather*}
	Note that for any $S\subseteq X$ such that $x^+$ exists for all $x\in S$, 
	\begin{align*}
	\bigveeplus S &= \min\{x \mid \forall_{y\in S}[ y < x ]\}\\
	&= \min\{x \mid \forall_{y\in S}[ y^+  \leqslant x ]\}\\
	&= \min\{x \mid \forall_{y\in S^+ } [y \leqslant x] \}\\
	&= \bigvee S^+ .\qedhere
	\end{align*}
	Furthermore, the following basic laws are self-evident:
	\begin{enumerate}[(i)]
		\item[(L1)] $x<x^+ $,
		\item[(L2)] there is no $z$ such that $x<z<x^+ $,
		\item[(L3)] $x<y \;\Leftrightarrow\; x^+ \leqslant y$,
		\item[(L4)] $x=\bigveeplus \{x\}\lc $ (for a total order),
		\item[(L5)] $x<y^+  \;\Leftrightarrow\;x\leqslant y $ (for a total order),
		\item[(L6)] $x^+ =y^+ \;\Leftrightarrow\; x=y$ (for a total order),
		\item[(L7)] $x<y \;\Leftrightarrow\; x^+ <y^+ $ (for a total order).
	\end{enumerate}
	
	The following (easy) lemma will also be useful:
	
	\begin{lemma}\label{lem:max}
		Let $(X,\leqslant)$ be a poset and let $S\subseteq X$. Then:
		\begin{enumerate}[(i)]
			\item If $S$ does not have a largest element, then $\bigvee S$ exists if and only if $\bigveeplus S$ exists, and when they exist they are equal, $\bigvee S=\bigveeplus S$. Conversely, if $\bigvee S=\bigveeplus S$ then $S$ does not have a largest element.  
			
			\item If $S$ has a largest element $x=\max S$, then $\bigveeplus S$ exists if and only if $x^+ $ exists, and when they exist, $\bigveeplus S=x^+$. Conversely, when $\leqslant$ is a total order, if $\bigveeplus S=x^+$ then $x=\max S$.
			
			\item If $\bigveeplus S\lc $ exists then $\bigveeplus S\lc =\min S$. Conversely, when $\leqslant$ is a total order, if $\min S$ exists, then $\min S=\bigveeplus S\lc $.
		\end{enumerate}
	\end{lemma}
	
	\begin{proof}
		~\begin{enumerate}[(i)]
			\item If $S$ has no largest element, then 
			\begin{gather*}
			    \{x\in X\mid \forall _{y\in S}[y\leqslant x] \}=\{x\in X\mid \forall _{y\in S}[y<x] \}
			\end{gather*}
			and so 
			\begin{gather*}
			    \bigvee S=\min \{x\in X\mid \forall _{y\in S}[y\leqslant x] \}
			\end{gather*}
			exists if and only if 
			\begin{gather*}
			    \bigveeplus S=\min \{x\in X\mid \forall _{y\in S}[y<x] \}
			\end{gather*}
			exists, and when they exist, they are equal.
			Now suppose $\bigvee S=\bigveeplus S$. Since $\bigveeplus S$ is never an element of $S$, we conclude that $S$ does not have a largest element.
			\item Let $x=\max S$. Then $\{x\}\uc =S\uc $, and so $x^+=\min(\{x\}\uc )$ exists if and only if $\bigveeplus S=\min(S\uc )$ exists, and when they exist, they are equal.
			Now suppose, in the case of total order, that $\bigveeplus S= x^+ $ for some $x\in X$. Then $S$ can only have elements that are strictly smaller than $x^+ $, and thus each element of $S$ is less than or equal to $x$ (L5). Also, since $x^+  =\bigveeplus S= \min(S\uc )$, it does not hold that $x\in S\uc $, and thus $x$ is not strictly larger than every element of $S$. We can conclude that $x=\max(S)$.
			
		\item Suppose $\bigveeplus S\lc $ exists. Since $S\subseteq (S\lc )\uc $, we must have $\bigveeplus S\lc \leqslant x$ for each $x\in S$. This together with the fact that $\bigveeplus S\lc $ cannot be an element of $S\lc $ forces $\bigveeplus S\lc $ to be an element of $S$. Hence $\bigveeplus S\lc =\min S$. Suppose now that $\leqslant$ is a total order and $\min S$ exists. Consider an element $x\in (S\lc )\uc $. Then $x$ is not an element of $S\lc $ and so we cannot have $x<\min S$. Therefore, $\min S\leqslant x$. This proves $\min S=\bigveeplus S\lc $.   \qedhere
		\end{enumerate}
	\end{proof}
	
	
	\begin{definition}\label{def:aos}
		An \emph{ordinal system} relative to a universe $\univ$ is a partially ordered set $\absto=(\absto,\leqslant)$ satisfying the following axioms:
		\begin{itemize}
			\item[(O1)] For all $X\subseteq \absto$, if $X\neq\varnothing$, then $X\lc \in \pow_\univ \absto$.
			\item[(O2)] $\bigveeplus X$ exists for each $X\in \pow_\univ \absto$.
		\end{itemize}
		We refer to elements of $\absto$ as \emph{ordinals}.
	\end{definition}

	Note that if $\absto$ is non-empty, then Axiom (O1) forces the universe $\univ$ to be non-empty as well. So for any ordinal $x\in\absto$ we have $\{x\}\in\pow_\univ \absto$ (Lemma~\ref{lem:fini}). Note also that Axiom (O1) is equivalent to its weaker form (the equivalence does not require (O2) and relies on Lemma~\ref{lem:subs}):
	\begin{enumerate}
		\item[(O1${}'$)] $\{x\}\lc \in \pow_\univ \absto$ for all $x\in \absto$.
	\end{enumerate}	
	Axiom (O2) implies that the mapping $x\mapsto x^+$ is a function $ \absto\to \absto$.
	We call this function the \emph{successor function} of the ordinal system $\absto$. For each $x\in\absto$, an element of $\absto$ that has the form $x^+$ is called a \emph{successor ordinal} and the \emph{successor of $x$}. We call an ordinal that is not a successor ordinal a \emph{limit ordinal}. 
	
	Recall that a poset is a \emph{well-ordered set} when each of its nonempty subsets has smallest element.
	
	\begin{theorem}\label{thm:wel}\label{thm:wo}
		A poset $(\absto,\leqslant)$ is an ordinal system relative to a universe $\univ$ if and only if it is a well-ordered set (and consecutively, a totally ordered set) satisfying (O1${}'$) and such that $X\uc \neq\varnothing$ for all $X\in\pow_\univ(\absto)$.
	\end{theorem}
	
	\begin{proof}
		Consider an ordinal system $\absto$ relative to a universe $\univ$. Let $X$ be a non-empty subset of $\absto$. Then $X\lc \in \pow_\univ \absto$ by (O1), and thus $\bigveeplus X\lc$ exists by (O2). 
		%
		By Lemma~\ref{lem:max}(iii), $\bigveeplus X\lc=\min X$. This proves the `only if' part of the theorem. Note that any well-ordered set is totally ordered: having a smallest element of a two-element subset $\{x,y \}$ forces $x$ and $y$ to be comparable. The `if' part is quite obvious.
	\end{proof}
	
	We will use this theorem often without referring to it. One of its consequences is that each non-empty ordinal system has a smallest element. We denote this element by $0$. Note that $0$ is a limit ordinal. Since by the same theorem an ordinal system is a total order, the properties (L4--7) above apply to an ordinal system. In particular, we get that the successor function is injective. We also get the following:
	
	\begin{lemma}\label{lem:limor}
	In an ordinal system, for an ordinal $x$ the following conditions are equivalent:
	\begin{enumerate}[(i)]
	\item $x$ is a limit ordinal. 
	
	\item $\{x\}\lc$ is closed under successors.
    
    \item $x=\bigvee\{x\}\lc$. 
	\end{enumerate}
	\end{lemma}
	
	\begin{proof}
	 We have $x=\bigveeplus\{x\}\lc$ for any ordinal $x$ (L4). If $x$ is a limit ordinal then for each $y<x$ we have $y^+<x$ (L3). So (i)${}\Rightarrow{}$(ii). If (ii) holds, by by (L1), we get that $\{x\}\lc$ does not have a largest element. So $\bigveeplus\{x\}\lc=\bigvee\{x\}\lc$ (Lemma \ref{lem:max}). This gives (ii)${}\Rightarrow{}$(iii). Suppose now $x=\bigvee\{x\}\lc$. If $x$ were a successor ordinal $x=y^+$, then by (L5), $y$ would be the join of $\{x\}\lc$. However, $x\neq y$ by (L1), and therefore, $x$ must be a limit ordinal. Thus, (iii)${}\Rightarrow{}$(i).
	\end{proof}
	
	The following theorem gives yet another way of thinking about an abstract ordinal number system.
	
	\begin{theorem}\label{thm:def2}
		A poset $(\absto,\leqslant)$ is an ordinal system relative to a universe $\univ$ if and only if (O1) holds along with the following axioms:
		\begin{itemize}
			\item[(O$2$a)] For all $X\in\pow_\univ \absto$, the join $\bigvee X$ exists.
			\item[(O$2$b)] $x^+$ exists for each $x\in\absto$.
		\end{itemize}
	\end{theorem}
	
	\begin{proof}
	This can easily be proved using (i) and (ii) of Lemma~\ref{lem:max}.
	\end{proof}
	
	Transfinite induction and recursion are well known for well-ordered sets. We formulate them here (in one particular form, out of many possibilities) in the case of ordinal systems, since we will use them later on in the paper. We have included our own direct proofs, for the sake of completeness, but we do not expect these proofs to have any new arguments that do not already exist in the literature. 
	
	\begin{theorem}[transfinite induction]\label{thm:ind1}
		Let $\absto$ be an ordinal system and let $X\subseteq \absto$ satisfy the following conditions:
		\begin{itemize}
			\item[(I1)] $X^+ \subseteq X$;
			\item[(I2)] for every limit ordinal $x$, if $\{x\}\lc \subseteq X$ then $x\in X$.
		\end{itemize}
		Then $X=\absto$.
	\end{theorem}
	
	\begin{proof}
		Since $\absto$ is well-ordered, if $\absto \setminus X$ is non-empty then it has a smallest element $y$. By (I1), $y$ cannot be a successor of any $z<y$ in $X$. By (I2), it also cannot be a limit ordinal. This is a contradiction since a limit ordinal is defined as one that is not a successor ordinal.
	\end{proof}
	
	Here is one of the immediate consequences of transfinite induction (the proof will require also the total order of an ordinal system, as well as the properties (L3) and (L4)):
	
	\begin{corollary}\label{cor:subordinal}
	Let $\absto$ be an ordinal system relative to a universe $\univ$ and let $X\subseteq \absto$ satisfy the following conditions:
		\begin{itemize}
			\item[(S1)] $X$ is down-closed in $\absto$, i.e., if $x<y\in X$ then $x\in X$, for all $x,y\in\absto$;
			\item[(S2)] $X$ is an ordinal system relative to $\univ$ under the restriction of the order of $\absto$. 
		\end{itemize}
		Then $X=\absto$.
	\end{corollary}
	
	\begin{theorem}[transfinite recursion]
		Let $\absto$ be an ordinal system and let $X=(X,V,T)$, where $X$ is a set, $T$ is a function $X\to X$ and $V$ is a function $\pow_\univ X\to X$. Then there exists a unique function $f\colon \absto\to X$ that satisfies the following conditions:
		\begin{itemize}
        	\item[(R1)] $f(x^+ ) = T(f(x))$ for any $x\in\absto$,
			\item[(R2)] $f(x) = V(\{f(y)\mid y<x \})$ for any limit ordinal $x$.
		\end{itemize}
	\end{theorem}
	
	\begin{proof}
	 For each ordinal $x\in\absto$, let $F_x$ be the set consisting of all functions 
	 \begin{gather*}
	 f_x \colon \{y\in\absto\mid y\leqslant x \}\to X
	 \end{gather*}
	 that satisfy the following conditions:
		\begin{enumerate}[(i)]
			\item $f(y^+ ) = T(f(y))$ for any successor ordinal $y^+ \leqslant x$, and
			\item $f(y) = V(\{f(z)\mid z<y \})$ for any limit ordinal $y\leqslant x$.
		\end{enumerate}
		Note that any function $f$ satisfying (R1--2) must have a subfunction in each set $F_x$. Also, for any $x\in\absto$, a function $f_x\in F_x$ must have a subfunction in each set $F_y$ where $y<x$. 
		We prove by transfinite induction that for each $x\in\absto$, the set $F_x$ contains exactly one function $f_x$.
		\begin{description}
	        \item[Successor case.] Suppose that for some ordinal $x$ and each $y\leqslant x$ there exists a unique function $f_y\in F_y$. Then $g=\uni\{f_x, \{(x^+ ,T(f_x(x)))\}\}$ satisfies (i,ii) and thus $g\in F_{x^+ }$.
			
			Now consider any function $g'\in F_{x^+ }$. Since $g$ and $g'$ must both have the unique function $f_x\in F_x$ as a subfunction, $g$ and $g'$ are identical on the domain $\{y\mid y\leqslant x \}$. But since $g$ and $g'$ both satisfy condition (i), we get that $g'(x^+ ) = T(f_x(x)) = g(x^+ )$ and thus $g'=g$.			
			\item[Limit case.] Suppose that for some limit ordinal $x$ and each  $y<x$ there exists a unique function $f_y\in F_y$. Then for any two ordinals $y<y'<x$, the function $f_y$ is a subfunction of $f_{y'}$, which is in turn a subfunction of every function in $F_x$. The relation 
		    \begin{gather*}
			g = \uni \{\{(x,V(\{f_y(y)\mid y<x \}))\},\uni \{f_y\mid y<x\}\}
			\end{gather*}
		    is then a function over the domain $\{y\mid y\leqslant x\}$ and moreover, it is easy to see that $g\in F_x$.
			
			Now consider any function $g'\in F_x$. Since for each ordinal $y<x$ the unique function $f_y\in F_y$ is a subfunction of both $g$ and $g'$, we see that they are identical on the domain $\{y\mid y<x \}$, and since they both satisfy condition (ii), the following holds:
			\begin{gather*}
			g(x) = V(\{g(y)\mid y<x \}) = V(\{g'(y)\mid y<x \}) = g'(x).
			\end{gather*}
			We can conclude that $g'=g$ and thus $g$ is the unique function in $F_{x}$.
		\end{description}
		We showed that for each ordinal $x$, exactly one function $f_x\in F_x$ exists. Construct a function $f\colon \absto\to X$ as follows: for each ordinal $x$,
		\begin{gather*}
		f(x) = f_x(x).
		\end{gather*}
		It satisfies (R1--2) because each $f_x$ satisfies (i,ii). Since any function satisfying (R1--2) must have a subfunction in each set $F_x$, we can conclude that $f$ is the unique function satisfying (R1--2).
	\end{proof}
	
	\section{Concrete ordinals}\label{sec:concrete}
	
	For a universe $\univ$, define a \emph{$\univ$-ordinal} to be a transitive set that belongs to the universe $\univ$ and is a well-ordered set under the relation
	\begin{gather*}
	x\inoreq y \quad\Leftrightarrow\quad [x\in y]\lor [x=y].
	\end{gather*}
	Thus, a $\univ$-ordinal is nothing but a usual von Neumann ordinal number \cite{von1923einfuhrung} that happens to be an element of $\univ$. In other words, it is a von Neumann ordinal number \emph{internal} to the universe $\univ$. Thus, at least one direction in the following theorem is well known. We include a full proof for completeness.
	
	\begin{theorem}\label{thm:ord}
		A set $O$ is the set of all $\univ$-ordinals if and only if the following conditions hold:
		\begin{enumerate}[(i)]
		\item $\varnothing\in O$ provided $\varnothing\in\univ$ ($\Leftrightarrow\univ\neq\varnothing$), and $O \subseteq \univ$, 
		\item $O$ is a transitive set, 
		
		\item $O$ is an ordinal system relative to $\univ$ under the relation $\inoreq$. 
		\end{enumerate}
	When these conditions hold, the incremented join of $X\in\pow_\univ O$ is given by $\bigveeplus X=\uni\{\uni X,X\}$.
	\end{theorem}

	\begin{proof}
		Let $O$ be the set of all $\univ$-ordinals. Then (i) follows easily from the definition of a $\univ$-ordinal. To show (ii), let $y\in x\in O$. We want to show that $y$ is a $\univ$-ordinal. Since $x\in\univ$ also $y\in\univ$ by (U1). Since $y\subseteq x$ and $x$ is well-ordered (under $\inoreq$), so is $y$. Next, we must prove that $y$ is transitive. Let $z\in y$. Then $z\in x$ since $x$ is transitive. We want to show $z\subseteq y$, so suppose $t\in z$. Then $t\in x$. By the fact that $\inoreq$ is an order on $x$, we must have $t\inoreq y$. We cannot have $t=y$, since $t\in z\in y$, and so $t\in y$. Next, we prove (iii). It is easy to see that the relation $\inoreq$ is a partial order on $O$: reflexivity is obvious, while transitivity follows from each element of $O$ being a transitive set. To prove (O1${}'$), consider $x\in O$. Then we have:
		\begin{gather*}
		\{x\}\lc  = \{y\in O\mid y\in x\}\subseteq x. 
		\end{gather*}
		Since $x$ is a $\univ$-ordinal, $x\in \pow_\univ O$ and so $\{x\}\lc \in \pow_\univ O$ (Lemma~\ref{lem:subs}). Note that by (ii) we actually have $x=\{x\}\lc $, although we did not need this to establish (O1${}'$). 
		
		For what follows, we will first establish the following:
        \begin{description}
        \item[Fact 1.] $x\cap y=\min(y\setminus x)$ for any two $\univ$-ordinals $x$ and $y$ such that $y\setminus x\neq\varnothing$. 
     \end{description}
		Let $a\in x\cap y$. Then $a\notin y\setminus x$. Hence $a\neq \min(y\setminus x)$. By the well-ordering of $y$, we then get that either $a\in\min(y\setminus x)$ or $\min(y\setminus x)\in a$. By transitivity of $x$ and the fact that $\min(y\setminus x)\notin x$, the second option is excluded. So $a\in\min(y\setminus x)$. This shows $x\cap y\subseteq \min(y\setminus x)$. Now suppose $a\in \min(y\setminus x)$. Then $a\in y$, by transitivity of $y$. If $a\notin x$, then $a\in y\setminus x$ which would give $ \min(y\setminus x)\inoreq a$, which is impossible. So $a\in x$. This proves $\min(y\setminus x)\subseteq x\cap y$. Hence $\min(y\setminus x)=x\cap y$.
		
		From Fact 1 we easily get the following:
		\begin{description}
        \item[Fact 2.] $x\inoreq y$ if and only if $x\subseteq y$, for any two $\univ$-ordinals $x$ and $y$.
     \end{description}
        This in turn implies that $O$ is totally ordered under $\inoreq$. Indeed, consider $x,y\in O$. If $x\neq y$ then either $x\setminus y$ or $y\setminus x$ is non-empty. Without loss of generality, suppose $y\setminus x$ is non-empty. By Fact 1,  $\min(y\setminus x)\subseteq x$. Then, by Fact 2, either $\min(y\setminus x)\in x$ or $\min(y\setminus x) = x$, and since $\min(y\setminus x)\in y\setminus x$ we can conclude that $x=\min(y\setminus x)$ and thus $x\in y$.
        
        We now show that $O$ is well-ordered under the relation $\inoreq$. Let $Y\subseteq O$ be nonempty. Take any $y\in Y$. If $y\cap Y=\varnothing$, then for all $x\in Y$ we have $x\notin y$ and thus $y=\min Y$ by total ordering. If $y\cap Y\neq \varnothing$, then $\min(y\cap Y)$ exists, since $y\cap Y\subseteq y$. For all $x\in Y\setminus y$, it holds that $x\notin y$ and thus  $y\subseteq x$ (by total ordering and Fact 2), which in turn implies $\min(y\cap Y)\in x$. We can conclude that $\min(y\cap Y)\in x$ for all $x\in Y$, and thus $\min Y = \min (y\cap Y)$.

		We will now complete the proof of (iii) and prove the last statement of the theorem simultaneously. Consider $X\in\pow_\univ O$. By (i) and Lemma~\ref{lem:com}, $X\in\univ$. Then $\uni \{\uni X,X\}\in\univ$. To show that $\uni \{\uni X,X\}$ is a $\univ$-ordinal, we need to prove that it is a transitive set, well-ordered under the relation $\inoreq$. Since each element of $X$ is transitive, we have $\uni \uni X\subseteq \uni X$, which implies transitivity of $\uni \{\uni X,X\}$: 
		\begin{gather*}
		\uni \uni \{\uni X,X\} = \uni \{\uni \uni X,\uni X\}=\uni X\subseteq \uni \{\uni X,X\}.
		\end{gather*}
		It follows from (ii) that each element of $\uni \{\uni X,X\}$ is a $\univ$-ordinal. So the fact that $\uni \{\uni X,X\}$ is well-ordered under the relation $\inoreq$ follows from the fact that $O$ is well-ordered under the same relation, as we have already proved. Thus, $\uni \{\uni X,X\}$ is a $\univ$-ordinal. Any $\univ$-ordinal $y$ such that $x\in y$ for every element $x\in X$ will have $\uni \{\uni X,X\}\subseteq y$ and thus $\uni \{\uni X,X\}\inoreq y$ by Fact 2. To conclude that $\uni \{\uni X,X\}$ is the incremented join of $X$, it remains to make a trivial remark that for each $x\in X$ we have $x\in \uni \{\uni X,X\}$.
		
		It remains to prove that if (i--iii) hold then $O$ is the set of all $\univ$-ordinals. Let $O$ be any set satisfying (i--iii). First we prove that every element of $O$ is a $\univ$-ordinal. Let $x\in O$. By (iii), the set $O$ is ordered under the relation $\inoreq$. In this ordered set, the set $\{x\}\lc $ consists of those elements of $x$ which are also elements of $O$. By (ii), this is all elements of $x$ and so $x=\{x\}\lc $. By (i) and (iii), $x\in\univ$.
		Furthermore, for any $y\in x$ the following holds:
		\begin{gather*}
		y = \{y\}\lc  \subseteq \{x\}\lc  = x.
		\end{gather*}
		This shows that $x$ is transitive. 
		Since $O$ is well-ordered under $\inoreq$ (by (iii) and Theorem~\ref{thm:wo}) and $x\subseteq O$, $x$ is also well-ordered under the same relation. Thus, every element of $O$ is a $\univ$-ordinal. 
		Now we need to establish that every $\univ$-ordinal is in $O$. We already proved that the set $\absto$ of all $\univ$-ordinals is an ordinal system, and since $O$ is a set of $\univ$-ordinals, $O\subseteq \absto$. The equality $O=\absto$ then follows from Corollary~\ref{cor:subordinal}. 
	\end{proof}
	
	\section{A Dedekind-style axiomatization}
	
	A \emph{limit-successor system} is a triple $X=(X,L,s)$ where $X$ is a set, $L$ is a partial function $L\colon \pow X\to X$ called the \emph{limit function} and $s$ is a function $s\colon X\to X$ called the \emph{successor function}. A \emph{successor-closed} subset of a limit-successor system $X$ is a subset $I$ of $X$ such that $sI\subseteq I$. For such subset, write $s^{-1}I$ to denote
	\begin{gather*}
	    s^{-1}I=\{x\in X\mid s(x)\in I\}
	\end{gather*}
	and $L^{-1}I$ to denote
	\begin{gather*}
	    L^{-1}I=\{A\in\mathsf{dom}L\mid L(A)\in I\}.
	\end{gather*}
    A subset $I\subseteq X$ is said to be \emph{closed} when 
    \begin{gather*}
        s^{-1}I\subseteq I\textrm{ and }\uni L^{-1}I\subseteq I.
    \end{gather*}
    It is not difficult to see that closed subsets form a topology on $X$; in fact, an Alexandrov topology. We denote the closure of a subset $I$ in this topology by $\overline{I}$. Recall that the corresponding `specialization preorder' given by 
    \begin{gather*}
        x\leqslant y\quad \Leftrightarrow\quad x\in\overline{\{y\}}
    \end{gather*}
    is a preorder (as it is for any Alexandrov topology) and that $x\in\overline{I}$ if and only if $x\leqslant y$ for some $y\in I$.
    
    We abbreviate the operator $s^{-1}$ composed with itself $m$ times as $s^{-m}$, with the $m=0$ case giving the identity operator. We write $
    s^{-\infty}$
    for the operator defined by
    \begin{gather*}
      s^{-\infty}I=\uni \{s^{-m} I\mid m\in\mathbb{N}\}
    \end{gather*}
    and $\uni{L^{-1}}$ for the operator $I\mapsto \uni L^{-1}I$.
    It is easy to see that the closure of a subset $I\subseteq X$ can be computed as 
    \begin{gather*}
        \overline{I}=\uni \{ s^{-\infty}\left[\uni{L^{-1}}s^{-\infty}\right]^k I\mid k\in\mathbb{N}\}.
    \end{gather*}
    This means that the specialization preorder `breaks up' into two relations $\leqslant_s$ and $\leqslant_L$, each determined by $s$ and $L$ alone, as explained in what follows. These relations are defined by:
    \begin{alignat*}{3}
        &x\leqslant_s y &&\quad\Leftrightarrow\quad \exists_{m\in\mathbb{N}}[s^m(x)=y] &&(\Leftrightarrow x\in s^{-\infty}\{y\})\\
        &x\leqslant_L y &&\quad\Leftrightarrow\quad \exists_{I}[[x\in I]\wedge[L(I)=y]]\qquad &&(\Leftrightarrow x\in \uni{L^{-1}}\{y\})
    \end{alignat*}
    We then have $x\leqslant y$ if and only if 
    \begin{gather*}
        x=z_0\leqslant_s z_1\leqslant_L z_2\leqslant_s z_3 \leqslant_L\dots z_{2k}\leqslant_s y
    \end{gather*}
    for some $z_0,\dots,z_{2k}\in X$, where $k$ can be any natural number $k\geqslant 0$. Note that $\leqslant_s$ is both reflexive and transitive, although the same cannot be claimed for $\leqslant_L$.
    
    \begin{definition}
    Given a universe $\univ$, a \emph{$\univ$-counting system} is a limit-successor system $(X,L,s)$ satisfying the following conditions:
    \begin{itemize}
    \item[(C1)] The domain of $L$ is the set of all successor-closed subsets $I\in\pow_\univ X$. 
    
    \item[(C2)] If $I$ and $J$ belong to the domain of $L$ and $\overline{I}=\overline{J}$, then $L(I)=L(J)$. 
    \end{itemize}
    \end{definition} 
    
    The structure above is the one that will be used for formulating the universal property of an ordinal system in the next section. In this section we give a characterization of  ordinal systems as counting systems having further internal properties.
    
    
    
    \begin{lemma}\label{lem:sL}
    Let $\univ$ be a universe and let $(X,L,s)$ be a $\univ$-counting system. Then $\leqslant_L$ is transitive and furthermore,
    \begin{gather*}
        x\leqslant_s y\leqslant_L z\quad \Rightarrow\quad x\leqslant_L z
    \end{gather*}
    for all $x,y,z\in X$.
    \end{lemma}
    
    \begin{proof}
    To prove transitivity, suppose $x\leqslant_L y$ and $y\leqslant_L z$. Then $x\in I$, $L(I)=y$, $y\in J$ and $L(J)=z$ for some successor-closed $I,J\in\pow_{\univ}X$. Then the union $\uni \{I,J\}\in\pow_{\univ}X$ (Lemma~\ref{lem:unio}) is also successor-closed, and hence it belongs to the domain of $L$ by (C1).  Since $L(I)\in J$, we get that $I\subseteq\overline{J}$. This implies that $\overline{\uni \{I, J\}}=\overline{J}$. By (C2), $L(\uni \{I, J\})=L(J)$. Having $x\in \uni \{I,J\}$ and  $L(\uni \{I, J\})=z$ means that $x\leqslant_L z$. This completes the proof of transitivity. To prove the second property, suppose $x\leqslant_s y\leqslant_L z$. Then $s^{m}(x)=y$ and $y\in J$ with $L(J)=z$, for some $m\in\mathbb{N}$ and successor-closed $J\in\pow_\univ X$. This proof follows a similar idea where we expand $J$, this time adding to it all elements of the form $s^{k}(x)$, where $k\in\{0,\dots,m-1\}$. The resulting set
    \begin{gather*}
    K= \uni \{\{x,s(x),\dots,s^{m-1}(x)\}, J\}
    \end{gather*}
    is clearly successor-closed and belongs to $\pow_\univ X$ (Lemmas \ref{lem:fini} and \ref{lem:unio}). Then, since $\overline{K}=\overline{J}$, we get that $L(K)=L(J)$ by (C2). This implies $x\leqslant_L z$.  
    \end{proof}
    
    This lemma gives that in a $\univ$-counting system $(X,L,s)$, for any $x,z\in X$ we have
    \begin{gather*}
        x\leqslant z\quad\Leftrightarrow\quad [x\leqslant_s z]\vee  \exists_y[x\leqslant_L y\leqslant_s z]
    \end{gather*}
    and hence the closure of a subset $I\subseteq X$ is given by
    \begin{gather*}
        \overline{I}=\uni \{s^{-\infty} I,\uni{L^{-1}}s^{-\infty} I\}.
    \end{gather*}
    
    \begin{theorem}\label{thm:spo}
    The specialization preorder of a $\univ$-counting system $(X,L,s)$ makes $X$ an ordinal system relative to $\univ$, provided the following conditions hold:
    \begin{itemize}
        \item[(C3)] $s^{-1}\{L(I)\}=\varnothing=I\cap \{s(L(I))\}$ for all $I$ such that $L(I)$ is defined. 
    
        \item[(C4)] $s$ is injective and $L$ has the property that if $L(I)=L(J)$ then $\overline{I}=\overline{J}$.
        
        \item[(C5)] $J=X$ for any successor-closed set $J$ having the property that $I\subseteq J\Rightarrow L(I)\in J$ every time $L(I)$ is defined. 
    \end{itemize}
    When these conditions hold, $s$ is the successor function of the ordinal system and $L(I)=\bigvee I=\bigveeplus I$ whenever $L(I)$ is defined; moreover, the limit ordinals are exactly the elements of $X$ of the form $L(I)$. Furthermore, 
    the closure of $I\in \pow_\univ X$ is given by $	\overline{I}=\left\{\bigveeplus I\right\}\lc
    $.
    Finally, any ordinal system relative to $\univ$ arises this way from a (unique) $\univ$-counting system satisfying (C3--5).
    \end{theorem}
    
Before proving the theorem, let us illustrate axioms (C1--5) in the case when $\univ$ is the universe of hereditarily finite sets (i.e., sets which are elements of a finite transitive set). Then $L$ is only defined on finite successor-closed sets. Injectivity of $s$ in (C4) forces every element $x$ of such set to have the property $\exists_{m\in\mathbb{N}\setminus\{0\}}[s^m(x)=x]$. At the same time, by (C3), such $x$ cannot lie in the image of $L$. So \begin{gather*}J=\{x\in X\mid \forall_{m\in\mathbb{N}\setminus\{0\}}[s^m(x)\neq x]\}\end{gather*} has the second property in (C5). Moreover, by injectivity of $s$ again, $J$ is also successor-closed. Then, by (C5), $J=X$ and so $L$ can only be defined on the empty set. With this provision, the triple $(X,L,s)$ becomes a triple $(X,0,s)$ where $0$ is the unique element in the image of $L$, $0=L(\varnothing)$. The axioms (C1--2) then trivially hold, while (C3--5) take the form of the axioms of Dedekind for a natural number system: 
\begin{itemize}
\item The first equality in (C3) states that $0$ does not belong to the image of $s$, while the second equality holds trivially.

\item (C4) just states that $s$ is injective.

\item (C5) becomes the usual principle of mathematical induction.
\end{itemize}
    
    \begin{proof}[Proof of Theorem~\ref{thm:spo}]
    Suppose the conditions (C1--5) hold. 
    
    \bigskip
    \textit{ Step 1. As a first step, we prove that the specialization preorder is antisymmetric, i.e., that it is a partial order.
    }\bigskip 
    
    For this, we first show that $\leqslant_L$ is `antireflexive': it is impossible to have  $x\leqslant_L x$. Indeed, suppose $x\in I$ and $L(I)=x$. Since $I$ is successor-closed by (C1), $s(x)\in I$. But then $s(x)\in I\cap \{s(L(I))\}$, which is impossible by (C3). 
    
    Next, we show antisymmetry of $\leqslant_s$. Suppose $x\leqslant_s z\leqslant_s x$ and $x\neq z$. Then we get that $s^k(x)=x$ for $k>1$. We will now show that this is not possible. In fact, we establish a slightly stronger property, which will be useful later on as well:
    \begin{description}
    \item[Property 0.] $s^k(x)\neq x$ for all $k>0$ and $x\in X$.
    \end{description}
    Actually, we have already established this property in the remark after the theorem. Here is a more detailed argument. Consider the set $J$ of all $x\in X$ such that $s^k(x)\neq x$ for all $k>0$. We will use (C5) to show that $J=X$. First, we show that $J$ is successor-closed. Let $y\in J$. Suppose $s^k(s(y))=s(y)$ for some $k>0$. Then by injectivity of $s$ (which is required in (C4)), $s^{k-1}(s(y))=y$, which is impossible. So $s(y)\in J$, showing that $J$ is successor-closed. Now let $I\subseteq J$ be successor-closed $I\in\pow_\univ X$ (by (C1), $I$ is such if and only if $L(I)$ is defined). Then $L(I)\neq s^k(L(I))$ for all $k>0$ by the first equality in (C3). So $L(I)\in J$ and we can apply (C5) to get $J=X$, as desired. 
    
    Antisymmetry of $\leqslant_s$ has thus been established.
    
    We are now ready to prove the antisymmetry of $\leqslant$. Suppose $x\leqslant z$ and $z\leqslant x$. There are four cases to consider:
    \begin{description}
        \item[Case 1.] $x\leqslant_s z\leqslant_s x$. Then $x=z$ by antisymmetry of $\leqslant_s$. 
        
        \item[Case 2.] $x\leqslant_L y\leqslant_s z\leqslant_s x$ for some $y$. By transitivity of $\leqslant_s$ and  Lemma~\ref{lem:sL}, in this case we get $y\leqslant_L y$, which we have shown not to be possible.
        
        \item[Case 3.] $x\leqslant_s z\leqslant_L y\leqslant_s x$ for some $y$. Similar to the previous case, in this case we get $y\leqslant_L y$, which is impossible.
        
        \item[Case 4.] $x\leqslant_L y\leqslant_s z\leqslant_L y'\leqslant_s x$ for some $y,y'$. In this case too we get the impossible $y\leqslant_L y$ (this case relies in addition on transitivity of $\leqslant_L$, established also in Lemma~\ref{lem:sL}).
    \end{description}
    
    We have thus shown that the specialization preorder is antisymmetric. We will now establish the following two properties, which will be useful later on.
    \begin{description}
        \item[Property 1.] If $x<y$ then $s(x)\leqslant y$, for all $x,y\in X$.
        
        \item[Property 2.] If $y<s(x)$ then $y\leqslant x$, for all $x,y\in X$.
    \end{description}
    To prove the first property, suppose $x<y$. There are two cases:
    \begin{description}
    \item[Case 1.]  $x\leqslant_s y$; then clearly $s(x)\leqslant y$ (as $x\neq y$).
    
    \item[Case 2.] $x\leqslant_L y'\leqslant_s y$ for some $y'$. Then $x\in I$ and $L(I)=y'$ for a successor-closed $I$, by (C1). So $s(x)\in I$ and thus $s(x)\leqslant_L y'$. With $y'\leqslant_s y$ this gives $s(x)\leqslant y$.
    \end{description}
    So in both cases we get $s(x)\leqslant y$, as required. To prove Property 2, suppose $y<s(x)$. We have again two cases:
    \begin{description}
    \item[Case 1.] $y\leqslant_s s(x)$. This with $y\neq s(x)$ gives $y\leqslant x$ by injectivity of $s$ from (C4).
    
    \item[Case 2.] $y\leqslant_L y'\leqslant_s s(x)$. Since $y'\neq s(x)$ by the first equality in (C3), by injectivity of $s$ from (C4) we must have $y'\leqslant_s x$. This will give $y\leqslant x$.  
    \end{description} We get the required conclusion in both cases.
    
    \bigskip
    \textit{ Step 2.
    Next, we want to prove that the specialization order is a total order.}
    \bigskip
    
    We will prove this by simultaneously establishing the following:
    \begin{description}
    \item[Property 3.]Let $y\in X$. If $x<y$ for all $x\in I$ such that $L(I)$ is defined, then necessarily $L(I)\leqslant y$.
    \end{description}
    Let $J$ be the set of all $x\in X$ such that for every $y\in X$ either $y\leqslant x$ or $x\leqslant y$. We will use (C5) to show that $J=X$. For this, we first prove that $J$ is successor-closed. Let $x\in J$. Consider any $y\in X$. Since $x\leqslant s(x)$, if $y\leqslant x$ then $y\leqslant s(x)$. If $x<y$ then by Property 1, $s(x)\leqslant y$. This proves that $J$ is successor-closed. Consider now $L(I)$, where $I\subseteq J$. To prove that $L(I)\in J$, we proceed as follows. Let $x\in X$. If $x\leqslant y$ for at least one $y\in I$, then $x\leqslant L(I)$, since $y\leqslant_L L(I)$. Thus, it suffices to prove that the set $K_I$ of all $x\in X$ such that if $y< x$ for all $y\in I$ then $L(I)\leqslant x$, is the entire $K_I=X$. This we prove using (C5). First, we show that $K_I$ is successor-closed. Suppose $x\in K_I$. If $y<s(x)$ for all $y\in I$, then by Property 2,  $y\leqslant x$ for all $y\in I$. If $x\in I$ then, since $I$ is successor-closed by (C1), we will have $s(x)\in I$, which will violate the assumption that $y<s(x)$ for all $y\in I$. So we get that $y<x$ for all $y\in I$. Then $L(I)\leqslant x$, since $x\in K_I$. This implies $L(I)\leqslant s(x)$, thus proving that $K_I$ is successor-closed. Now let $H\subseteq K_I$ be such that $L(H) $ is defined. Suppose $y<L(H) $ for all $y\in I$. From the first equality in (C3) we get that $y\leqslant_L L(H) $ for each $y\in I$. So for each $y\in I$, there is $G_y$ such that $y\in G_y$ and $L(G_y)=L(H) $. This implies that $\overline{G_y}= \overline{H}$ for each $y\in I$ (by (C4)), and so $I\subseteq \overline{H}$. Since $I\subseteq J$, each element of $H$ is comparable with each element of $I$. If for every $h\in H$ we have $y_h\in I$ such that $h\leqslant y_h$, then $H\subseteq \overline{I}$. This would then give $\overline{I}=\overline{H}$, and so by (C2), $L(I)=L(H) $, showing that $L(I)\leqslant L(H) $, as desired. In contrast, if there is $h\in H$ such that $y<h$ for every $y\in I$, then (since $H\subseteq K_I$) $L(I)<h$. This together with $h\leqslant_L L(H) $ will give $L(I)\leqslant L(H) $. We have thus shown that $K_I$ has the required properties in order for us to apply (C5) to conclude that $K_I=X$. This then shows that $L(I)\in J$, and so $J$ has the required properties to conclude that $J=X$. The proof of the specialization order being a total order is then complete. At the same time, since $J=X$ and for each $I\subseteq J$ such that $L(I)$ is defined, $K_I=X$, we have also established Property 3. 
    
    \bigskip\textit{ Step 3. We now show that $L(I)=\bigvee I=\bigveeplus I$ whenever $L(I)$ is defined and $s(x)=x^+$ for all $x\in X$.
    }\bigskip 
    
    Properties 0 and 3 show  that $L(I)$ is the join of $I$, for any $I$ such that $L(I)$ is defined. Indeed, if $x\leqslant y$ for all $x\in I$, then for each $x\in I$, also $s(x)\leqslant y$. Since $x<s(x)$, as clearly $x\leqslant s(x)$ and by Property 0, $x\neq s(x)$, we get: $x<y$ for all $x\in I$. Then by Property 3, $L(I)\leqslant y$ thus showing that $L(I)$ is a join of $I$. Furthermore, the property $x<s(x)$ together with Property 1 implies that $s(x)=x^+$, for each $x\in X$. Thus, once we prove that $X$ is an ordinal system under the specialization order, we have that $s$ is its successor function and $L$ is given by join. Furthermore, when $L(I)$ is defined, $I$ is successor-closed and so it cannot have a largest element, by Property 0. Then the join $L(I)$ of $I$ must also be the incremented join of $I$ (Lemma~\ref{lem:max}). 
    
    \bigskip 
    \textit{Step 4. We show that $X$ is an ordinal system under the specialization order where limit ordinals are exactly the elements of the form $L(I)$.}
    \bigskip 
    
    Consider the set $J$ of all $x\in X$ such that $\{x\}\lc \in\pow_\univ X$. If $x\in J$, then by Property 2, $\{s(x)\}\lc =\uni\{\{x\}\lc ,\{x\}\}\in \pow_\univ X$ (Lemmas \ref{lem:fini} and \ref{lem:unio}) and so $s(x)\in J$. Suppose $I\in J$ is such that $L(I)$ is defined. Since the specialization order is a total order and $L(I)$ is the join of elements in $I$, we have
    \begin{gather*}
        \{L(I)\}\lc =\uni \{ \{x\}\lc\mid x\in I\} \in \pow_\univ X \quad\textrm{ ((C1) \& Lemma \ref{lem:unio})}
    \end{gather*}
    and so $L(I)\in J$. By (C5), $J=X$. 
    
    To prove that $X$ is an ordinal system under the specialization order, it remains to prove that for any $Y\in \pow_\univ X$, the incremented join of $Y$ exists in $X$. If $Y$ has a largest element, then the successor of that element is the incremented join of $Y$ (Lemma~\ref{lem:max}). If $\univ$ does not contain an infinite set, then $Y$ is finite and so it has a largest element. Now consider $Y\in\pow_\univ X$ that has no largest element, with $\univ$ containing an infinite set. We define:
    \begin{gather*}
        s^\infty Y = \{ s^n (x)\mid [x\in Y]\wedge[n\in\mathbb{N}]\}.
    \end{gather*}
    This is of course the closure of $Y$ under $s$. Then $s^\infty Y\in\pow_\univ X$ (Lemmas \ref{lem:imag} and \ref{lem:produ}) and so $L(s^\infty Y)$ is defined by (C1). Since $Y\subseteq s^\infty Y$, it holds that $y<L(s^\infty Y)$ for all $y\in Y$. Let $y\in Y$. We prove by induction on $n$ that for each $n\in\mathbb{N}$, we have $s^n(y)<z$ for some $z\in Y$. For $n=0$, this follows from the fact that $Y$ does not have a largest element. Suppose $s^n(y)<z$ for some $z\in Y$. Then $s^{n+1}(y)=s(s^n(y))\leqslant z$. Since $z$ cannot be the largest element of $Y$, we must have $z<z'$ for some $z'\in Y$. Then $s^{n+1}(y)<z'$. What we have shown implies that the incremented join $L(s^\infty Y)$ of $s^\infty Y$ is also the incremented join of $Y$.
    
    We have thus proved that the specialization (pre)order of a $\univ$-counting system satisfying (C3--5) makes it an ordinal system relative to $\univ$, with $s$ as its successor function and $L$ given equivalently by join and by incremented join. This also establishes that if an ordinal system relative to $\univ$ arises this way from a $\univ$-counting system satisfying (C3--5), then this $\univ$-counting system is unique. We now prove the existence of such a $\univ$-counting system. Actually, before doing that, note that by (C3), no element of $X$ of the form $L(I)$ can be a successor ordinal, and so it must be a limit ordinal. Conversely, for a limit ordinal $x$ we have $x=L(\{x\}\lc)$ (Lemma~\ref{lem:limor}). This shows that limit ordinals are precisely the ordinals of the form $L(I)$.
    
    For an ordinal system $\absto$ relative to $\univ$, consider the limit-successor system $(\absto,\bigvee,\_^+)$, where $\bigvee$ is the usual join restricted on a domain as required by (C1). 
    
    \bigskip 
    \textit{Step 5. We show that (C1--5) hold for the limit-successor system $(\absto,\bigvee,\_^+)$ and that the corresponding specialization order matches with the order of $\absto$. In this step we show as well that  $\overline{I}=\left\{\bigveeplus I\right\}\lc$ holds for each $I\in\pow_\univ\absto$.}\bigskip 
    
    By Theorem~\ref{thm:def2}, $L$ is indeed defined over the entire domain required in (C1). To prove (C2), first we establish that
    	\begin{gather*}
    	\overline{I}=\left\{\bigveeplus I\right\}\lc 
    	\end{gather*} 
    	for each $I\in \pow_\univ\absto$. It is easy to see that $\{\bigveeplus I\}\lc $ is closed, so $\overline{I}\subseteq \{\bigveeplus I\}\lc $. To show $\{\bigveeplus I\}\lc\subseteq \overline{I}$, let $x\in \{\bigveeplus I\}\lc$. We have well-ordering and hence total order by Theorem \ref{thm:wel}.  Then $x<\bigveeplus I$ and so $x\leqslant y\in I$ for some $y$. Consider 
    	\begin{gather*}
    	y'=\min\{y\in \overline{I}\mid x
    	\leqslant y\}.
    	\end{gather*}
    	We consider two cases:
    	\begin{description}
    \item[Case 1.] $y'$ is a successor ordinal. Then $y'=y''^+$ for some ordinal $y''$. Since $y'\in \overline{I}$, we must have $y''\in 
    \overline{I}$. Then $y''<x$ and so $y'\leqslant x$ (L3). This gives $x=y'$ and so $x\in \overline{I}$.  
    
    \item[Case 2.] $y'$ is a limit ordinal. Then $\{y'\}\lc $ is successor-closed. Furthermore, we have
    	\begin{gather*}
	y'=\bigveeplus\{y'\}\lc =\bigvee(\{y'\}\lc )^+=\bigvee\{y'\}\lc.
    	\end{gather*} 
    	By closure of $\overline{I}$, we get $\{y'\}\lc \subseteq \overline{I}$.
    	Since $y<x$ for all $y<y'$, we get $y'\leqslant x$. This gives $x=y'$ and so $x\in\overline{I}$.
    \end{description}
    	
    	 We have thus established that the equality $\overline{I}=\{\bigveeplus I\}\lc $ holds
    	for each $I\in\pow_\univ \absto$. From this it follows that the specialization preorder matches with the order of $\absto$. We then get that (C2) holds by the fact that if down-closures of two subsets of a poset are equal, then so are their joins. Thus $(\absto,\bigvee,\_^+)$ is a $\univ$-counting system.
    	
    	It remains to show that (C3--5) hold. Consider a successor-closed  $I\in\pow_\univ\absto$. $I$ has no maximum element thanks to (L1) and thus, $\bigvee I = \bigveeplus I$ by Lemma~\ref{lem:max}. By the same lemma, $\bigveeplus I$ cannot be a successor if $I$ has no maximum. Thus we have $(\_^+)^{-1}\{\bigvee I \} = \varnothing$, which is the first part of (C3). Since $(\bigvee I)^+ > \bigvee I \geqslant x$ for each $x\in I$, we have that $(\bigvee I)^+\notin I$, which means the second part of (C3) also holds, i.e. $I\cap \{(\bigvee I)^+\} = \varnothing$.
    		
    	We already know that $\_^+$ is injective, so to see that (C4) holds, consider another successor-closed $J\in\pow_\univ\absto$. If $\bigvee I = \bigvee J$, then 
    	\begin{gather*}
    	\overline{I} = \{\bigveeplus I \}\lc  = \{\bigvee I \}\lc  = \{\bigvee J \}\lc  = \{\bigveeplus J \}\lc  = \overline{J}.
    	\end{gather*}
    	Thus (C4) holds. Finally, consider a successor-closed subset $J$ of $\absto$ where $\bigvee I\in J$ for all successor-closed subsets $I$ of $J$ such that $I\in\pow_\univ\absto$. Then $J=\absto$ if it satisfies (I1) and (I2) in our formulation of transfinite induction. We check both:
    	\begin{itemize}
    		\item[(I1)] $J^+\subseteq J$ follows from the fact that $J$ is successor-closed.
    		\item[(I2)] Let $x$ be a limit ordinal such that $\{x\}\lc\subseteq J$. Then $x=\bigvee\{x\}\lc \in J$ (Lemma~\ref{lem:limor}).
    	\end{itemize}
    	Since both of these conditions hold, we can conclude that $J=\absto$, and thus (C5) holds. This completes the proof.
    \end{proof}

    \section{The corresponding universal property}
    
    Given two $\univ$-counting systems $(X_1,L_1,s_1)$ and $(X_2,L_2,s_2)$, a function $f\colon X_1\to X_2$ that preserves the successor function ($fs_1=s_2f$) automatically preserves successor-closed subsets, so for any successor-closed $I\in\pow_\univ X_1$, both sides of the equality
    \begin{gather*}
        L_2(fI)=f(L_1(I))
    \end{gather*}
    are defined (Lemma \ref{lem:imag}). When this equality holds for any such $I$, we say that $f$ is a \emph{morphism} of $\univ$-counting systems and represent $f$ as an arrow
    \begin{gather*}
        f\colon (X_1,L_1,s_1)\to (X_2,L_2,s_2).
    \end{gather*}
    It is not difficult to see that $\univ$-counting systems and morphisms between them form a category, under the usual composition of functions. Isomorphisms in this category are bijections between $\univ$-counting systems which preserve both succession and limiting. Call a $\univ$-counting system an \emph{ordinal $\univ$-counting} system when conditions (C3--5) hold. Clearly, the property of being an ordinal $\univ$-counting system is stable under isomorphism of $\univ$-counting systems. By Theorems~\ref{thm:ord} and \ref{thm:spo}, an ordinal $\univ$-counting system exists and is given by the $\univ$-ordinals. We will now see that ordinal $\univ$-counting systems are precisely the initial objects in the category of $\univ$-counting systems.     
    
    \begin{theorem}\label{thm:unive} For any universe $\univ$, a $\univ$-counting system is an initial object in the category of $\univ$-counting systems if and only if it is an ordinal $\univ$-counting system.
    \end{theorem}
    
    \begin{proof}
        Since we know that an ordinal $\univ$-counting system exists (Theorem \ref{thm:ord}) and that the property of being an ordinal $\univ$-counting system is stable under isomorphism, it suffices to show that any ordinal $\univ$-counting system is an initial object in the category of $\univ$-counting systems. By Theorem~\ref{thm:spo}, an ordinal $\univ$-counting system has the form  $(\absto,\bigvee,\_^+)$, where $\absto$ is an ordinal system relative to $\univ$ and $\bigvee$ is the join defined for exactly the successor-closed subsets $I\in\pow_\univ \absto$ in the $\univ$-counting system. 
        
        For any $\univ$-counting system $(X,L,s)$, if a morphism $(\absto,\bigvee,\_^+)\to (X,L,s)$ exists, it must be the unique function $f$ defined by the transfinite recursion
    	\begin{enumerate}[(i)]
    		\item $f(x^+) = s(f(x))$ for any $x\in\absto$;
    		\item $f(x) = 
    		L(\{f(y)\mid y<x \})$ for any limit ordinal $x$ (Lemma \ref{lem:limor}).
    	\end{enumerate}
    	We now prove that the function $f$ defined by the recursion above is a morphism. It preserves succession by (i). Consider $I\in\pow_\univ \absto$ closed under successors. Then $\bigvee I$ is a limit ordinal and
    	$
    	\overline{I} = \{\bigvee I \}\lc
    	$, by Theorem~\ref{thm:spo}.
    	By definition of $f$, we then have
    	\begin{gather*}
    	f(\bigvee I) = L(\{f(y)\mid y<\bigvee I \}) = L(f\overline{I}).
    	\end{gather*}
    	We will now prove $\overline{f\overline{I}}= \overline{fI}$.
    	We clearly have $fI\subseteq \overline{f\overline{I}}$, so it suffices to show that $f\overline{I}\subseteq \overline{fI}$. This is equivalent to showing $\overline{I}\subseteq f^{-1}\overline{fI}$, which would follow if we prove $f^{-1}\overline{fI}$ is closed. If $x^+\in f^{-1}\overline{fI}$, then $s(f(x))=f(x^+)\in \overline{fI}$. Therefore, $f(x)\in \overline{fI}$ and so $x\in f^{-1}\overline{fI}$. If $\bigvee J\in f^{-1}\overline{fI}$, then (as $\bigvee J$ is a limit ordinal by Theorem \ref{thm:spo})
    	\begin{gather*}
    	L(\{f(y)\mid y<\bigvee J\})=f(\bigvee J)\in \overline{fI},
    	\end{gather*}
    	which implies $\{f(y)\mid y<\bigvee J\}\subseteq \overline{fI}$. This gives $J\subseteq \{\bigvee J\}\lc\subseteq f^{-1}\overline{fI}$. Note that we have the first of these two subset inclusions due to the fact that $\bigvee J=\bigveeplus J$ thanks to Theorem \ref{thm:spo}. This proves that $f^{-1}\overline{fI}$ is closed. So $\overline{f\overline{I}}= \overline{fI}$.  We therefore get $f(\bigvee I) = L(f\overline{I}) = L(fI)$, showing that $f$ is indeed a morphism $(\absto,\bigvee,\_^+) \to (X,L,s)$. 
    \end{proof}

        Consider the case when every element in $\univ\neq \varnothing$ is a finite set (e.g., $\univ$ could be the universe of hereditarily finite sets). Then every triple $(X,0,s)$, where $X$ is a set, $s$ is a function $s\colon X\to X$, and $0\in X$, can be seen as a $\univ$-counting system for the same $s$, with $L(I)=0$ for each finite $I$. A morphism $f\colon (X_1,0_1,s_1)\to (X_2,0_2,s_2)$ between such $\univ$-counting systems is a function $f\colon X_1\to X_2$ such that $s_2f=fs_1$ and $f(0_1)=0_2$. The natural number system $(\mathbb{N},0,s)$, with its usual successor function $s(n)=n+1$, is an initial object in the category of such $\univ$-counting systems. The theorem above presents the natural number system as an initial object in the category of all $\univ$-counting systems. It is not surprising that the natural number system is initial in this larger category too, since the empty set is the only finite successor-closed subset of $\mathbb{N}$.
        


\bibliographystyle{abbrv}
\bibliography{bibliography}
	










	
\end{document}